\documentclass{elsarticle}

\usepackage{amssymb}
\usepackage{amsthm}
\usepackage{mathtools}
\usepackage{hyperref}
\hypersetup{
  colorlinks   = true, 
  urlcolor     = blue, 
  linkcolor    = blue, 
  citecolor   = red 
}

\usepackage{nicefrac}
\usepackage{enumerate}
\usepackage{color}
\usepackage{url}
\usepackage{mathrsfs}

\newcommand{\K}{\mathcal{K}}

\newcommand{\R}{\mathbb{R}}

\DeclareMathOperator*{\supp}{supp}
\DeclareMathOperator*{\dist}{dist}

\newcommand{\lp}{L^p(\Omega)}
\newcommand{\wspr}{W^{s,p}(\R^n)}
\newcommand{\wsp}{W^{s,p}(\Omega)}
\newcommand{\wwsp}{\widetilde{W}^{s,p}(\Omega)}
\newcommand{\dwsp}{{W}^{-s,p^\prime}(\Omega)}

\newcommand{\mwsp}{\mathcal{W}^{s,p}(\Omega)}

\newtheorem{teo}{Theorem}[section]
\newtheorem{lem}[teo]{Lemma}

\newtheorem{pro}[teo]{Proposition}
\theoremstyle{remark}
\newtheorem{re}[teo]{Remark}

\theoremstyle{definition}

\journal{Journal of Differential Equations}

\begin{document}

\begin{frontmatter}
\title{A Hopf's lemma and a strong minimum principle for the fractional 
$p$-Laplacian}
\author[LDP]{Leandro M. Del Pezzo}
\ead{ldelpezzo@utdt.edu}

\author[AQ]{Alexander Quaas}
\ead{alexander.quaas@usm.cl}
%
\address[LDP]{
CONICET and Departamento  de Matem{\'a}tica y Estadística
Universidad Torcuato Di Tella
Av. Figueroa Alcorta 7350 (C1428BCW)
C. A. de Buenos Aires - Argentina
}

\address[AQ]{
Departamento de Matem{\'a}tica, Universidad 
T\'ecnica Federico Santa Mar\'ia Casilla V-110, Avda. 
Espa\~na, 1680 -- Valpara\'iso, CHILE.}

\begin{abstract}
	Our propose here is to provide a Hopf lemma and a 
	strong minimum principle for weak supersolutions
	of
	\[
		(-\Delta_p)^s u= c(x)|u|^{p-2}u \quad \text{ in } \Omega
	\]
	where $\Omega$ is an open set of $\mathbb{R}^N,$ $s\in(0,1),$
	$p\in(1,+\infty),$ $c\in C(\overline{\Omega})$ 
	and $(-\Delta_p)^s$ is the fractional
	$p$-Laplacian. 
\end{abstract}



\end{frontmatter}

\section{Introduction}
	
	It is well known that the Hopf's lemma is one of the most useful and 
	best known tool in the study of partial differential equations. 
	Just to name a some of its applications, 
	this lemma is crucial in the proofs of the strong maximum principle,
	and the anti-maximum principle and in the moving plane method.
	For a review on the topic in the local case, 
	see for instance \cite{MR2025185,MR2356201} 
	and the references therein. 
	\medskip
	
	Our propose here is to provide a Hopf lemma 
	and a strong minimum principle for the fractional $p$-Laplacian
	\[
		(-\Delta_p)^su(x)\coloneqq
		2\mathcal{K}(s,p,N)
		\lim_{\varepsilon\to0^+}
		\int_{\R^N\setminus B_\varepsilon(x)}\!\!\!\!\!\!\!\!\!\!\!\!\!\!
		\!\!\!\!\dfrac{|u(x)-u(y)|^{p-2}(u(x)-u(y))}{|x-y|^{N+sp}}dy\,
		x\in\R^N
	\]
	where $p\in(1,\infty),$ $s\in(0,1),$ and 
	$\mathcal{K}(s,p,N)$ is a normalization factor. The fractional 
	$p-$Laplacian is a nonlocal version of the $p-$Laplacian 
	and is an extension of the fractional Laplacian ($p= 2$).
	
	\medskip
	
	In the last few years, the nonlocal operators have taken 
	relevance because they arise in a number of applications in many 
	fields, for instance,  game theory, 
	finance, image processing, L\'evy processes, and optimization, see  
	\cite{Caffarelli2012,MR2042661,MR2480109,MR1406564,MR0521262}
	and the references therein. From of the mathematical point of view,
	the fractional $p-$Lapalcian has a great 
	attractive since two phenomena are present in it: the
	nonlinearity of the operator and its nolocal character. 
	See for instance \cite{DelPezzoQuaas,brasco2014second,MR3411543,
	MR3542614,IMS,LL,MR3491533,MR3456825,MR3483598} and the references
	therein.
	
	
	\subsection{Statements of the main results}
		Before starting to state our results we need to introduce the 
		theoretical framework for them.
		
		\medskip
		
		Throughout this paper, $\Omega$ is an open set of 
		$\R^N,$ 
		$s\in(0,1),$ $p\in(1,\infty)$ and 
		to simplify notation, we omit the constant $\K(s,p,N).$ 
		From now on, given a subset $A$ of $\R^N$ we set $A^c=\R^N\setminus A,$
	    and $A^2=A\times A.$
		
		\medskip
		
		The fractional Sobolev spaces $\wsp$
		is defined to be the set of functions $u\in\lp$ such that
		\[
			|u|_{\wsp}^p\coloneqq
			\int_{\Omega^2}
			\dfrac{|u(x)-u(y)|^p}{|x-y|^{N+sp}}\, dxdy<\infty.
		\]
		The fractional Soblev spaces admit the following norm
		\[
			\|u\|_{\wsp}\coloneqq\left(\|u\|_{\lp}^p+|u|_{\wsp}^p
			\right)^{\frac1p},
		\]
		where
		\[
			\|u\|_{\lp}^p\coloneqq\int_\Omega |u(x)|^p\, dx.
		\]
		The space $\wspr$ is defined similarly.
	
		\medskip
	
		We will denote by $\wwsp$ the space of all 
		$u\in\wsp$ such that $\tilde{u}\in\wspr,$ 
		where $\tilde{u}$ is the extension by 
		zero of $u.$ The dual space of $\wwsp$ is 
		denoted by $\dwsp$ and the 
		corresponding dual pairing is denoted by 
		$\langle \cdot , \cdot \rangle.$
		
		\medskip
		
		The space $\mathcal{W}^{s,p}(\Omega)$ is the space of 
		all function $u\in L^{p}_{loc}(\R^N)$ such that
		for any bounded $\Omega^\prime\subseteq\Omega$
		there is an open set $U\supset\supset\Omega^\prime$ so that
		$u\in W^{s,p}(U),$ and 
		\[
			[u]_{s,p}\coloneqq\int_{\R^N}\dfrac{|u(x)|^{p-1}}
			{(1+|x|)^{N+sp}}dx<\infty.
		\]
		
		\begin{re}
			Suppose that $\Omega$ is bounded. 
			Then $u\in \mwsp$ if only if there is an open set 
			$U\supset\supset\Omega$ such that $u\in W^{s,p}(U),$ and
			$[u]_{s,p}<\infty.$  In addition,
			$\wwsp\subset\mwsp.$ 
		\end{re}

%
	
%
		
		\medskip
		
		Further informations on fractional Sobolev spaces and 
		many references may be found in 
		\cite{Adams, DD, DNPV, Grisvard,IMS}. 

		\medskip
		
		Now, let us introduce our notion of weak super(sub)-solution.
		Given $f\in \dwsp,$  we say that 
		$f\ge(\le)0$ if for any $\varphi\in\wwsp,$ $\varphi\ge0$
		we have that $\langle f,\varphi \rangle \ge (\le)0.$
		
		\medskip
		
		When  $\Omega$ is bounded,  we say that $u\in\mwsp$ is a weak  
		super(sub)-solution of $(-\Delta_p)^su=f$ in $\Omega$ if
		\[
			\int_{\R^{2N}}\dfrac{(u(x)-u(y))^{p-1}
			(\varphi(x)-\varphi(y))}
			{|x-y|^{N+sp}}
			dxdy\ge(\le)\langle f,\varphi\rangle
		\] 
		for each $\varphi\in\wwsp,\varphi\ge0.$ 
		
		\medskip
		
		When $\Omega$ is unbounded, we say that $u\in\mwsp$ is a weak  
		super(sub)-solution of $(-\Delta_p)^su=f$ in $\Omega$ if
		for all bounded open set $\Omega^{\prime}\subset \Omega$
		we have that $u$ is a weak  
		super(sub)-solution of $(-\Delta_p)^su=f$ in $\Omega^\prime.$
	
		\medskip
		
		In both cases, $u$ is a weak solution of 
		$(-\Delta_p)^su=f$ in $\Omega$ if $u$ is a super and sub-solution
		of $(-\Delta_p)^su=f.$
		
		\medskip
		
		Given a function {$c \in L^{1}_{loc}(\Omega),$} we say that $u\in\mwsp$ is a weak  
		super(sub)-solution of $(-\Delta_p)^su=c(x)|u|^{p-1}u$ in $\Omega$ if
		$f=c(x)|u|^{p-2}u\in \dwsp,$ and $u\in\mwsp$ is a weak  
		super(sub)-solution of $(-\Delta_p)^su=f$ in $\Omega$. Finally,  we say that 
		$u$ is a weak solution of $(-\Delta_p)^su=c(x)|u|^{p-2}u$ in $\Omega$ if $u$ 
		is a super and sub-solution
		of $(-\Delta_p)^su=c(x)|u|^{p-2}u.$
		
		\medskip

		Our first result is the following minimum principle. 
	
		\begin{teo}\label{teo:mprinciple} 
			Let {$c \in L^{1}_{loc}(\Omega)$} be a non-positive function 
			and  $u\in\mwsp$ be a weak super-solution of  
			\begin{equation}\label{eq:ecC}
				(-\Delta_p)^s u = c(x)|u|^{p-2}u \quad \mbox{ in }\Omega.
			\end{equation}
 		
			\begin{enumerate}
				\item If $\Omega$ is bounded,  
					and $u\ge0$ a.e. in $\Omega^c$ then either 
					$u>0$ {a.e.} in $\Omega$ or $u=0$ a.e. in $\R^N.$
				\item  If  $u\ge0$ a.e in $\mathbb{R}^N$ then either 
					$u>0$ {a.e.} in $\Omega$ or $u=0$ a.e. in $\R^N.$
			\end{enumerate}  				
		\end{teo}
		
		\begin{re}\label{re:notnecassump}
			Observe that $c\in L^1_{loc}(\Omega)$ then 
			$c(x)\ge c^-(x)=\min\{0,c(x)\}\in L^1_{loc}(\Omega).$ 
			Then, if $u$ is 
			a  weak super-solution of \eqref{eq:ecC} 
			and $u\ge0$ a.e. in $\R^N$ then $u$ is
			also a weak super-solution of
			\[
				(-\Delta_p)^su=c^-(x)u^{p-1}\quad \text{ in } \Omega.
			\]
			Then, by Theorem \ref{teo:mprincipleVis}, $u>0$ a.e. in 
			$\Omega$ or $u=0$ a.e. in $\R^N.$ That is, in the case
			$u\ge0$ a.e. in $\R^n,$ the non-positivity assumption on
			the function $c$ is not necessary.
		
			In fact, the previous result also holds 
			for all measurable function $c$ for which there is 
			$d\in L^1_{loc}(\Omega)$ such that $c\ge d$ a.e. in 
			$\Omega.$ 
		\end{re}
	

		Under the assumption that $c$ and $u$ also are continuous function, 
		by the properties that all continuous weak super-solutions 
		are viscosity super-solutions and using a test function,
		we can remove ``a.e" in the statement of our previous theorem.
		For more details, see Section \ref{Preli}.
	
		\begin{teo}\label{teo:mprincipleVis} 
			Let {$c \in C(\overline{\Omega})$} be a non-positive function 
			and  $u\in\mwsp\cap C(\overline{\Omega})$ be a weak 
			super-solution of  
			\eqref{eq:ecC}. 
			\begin{enumerate}
				\item If $\Omega$ is bounded,  
					and $u\ge0$ a.e. in $\Omega^c$ then either 
					$u>0$ in $\Omega$ or $u=0$ a.e. in $\R^N.$
				\item  If  $u\ge0$ a.e. in $\mathbb{R}^N$ 
					then either $u>0$ in $\Omega$ or $u=0$ a.e. 
					in $\R^N.$
			\end{enumerate}  		
		\end{teo}
	
		\bigskip
	
		Lastly, we show our Hopf lemma.
	
		\begin{teo}\label{teo:Hopf}
	 		Let $\Omega$ satisfy the interior ball condition in 
	 		$x_0\in\partial\Omega$, $ c\in C(\overline{\Omega}),$
	 		and $u\in \mwsp\cap C(\overline{\Omega})$ 
	 		be a  weak super-solution of \eqref{eq:ecC}. 
			\begin{enumerate}
				\item If $\Omega$ is bounded, 
					$c(x)\le 0$ in $\Omega$ 
					and $u\ge0$ a.e. in $\Omega^c$ then either 
					$u=0$ a.e. in $\R^N$ or 
				\begin{equation}\label{eq:Hopf}
					\liminf_{B_R\ni x\to x_0}
					\dfrac{u(x)}{\delta_R(x)^s}>0		
				\end{equation}
				where $B_R\subseteq \Omega$ and $x_0\in \partial B_R$ 
				and $\delta_R(x)$  is distance from $x$ to
				$B_R^c.$
				\item  If  $u\ge0$ a.e. in $\mathbb{R}^N$ then either 
					$u=0$  a.e. in $\R^N$ or \eqref{eq:Hopf} holds.
			\end{enumerate}  		
		\end{teo}
	
		\medskip
		
		Now, we give a brief resume about the Hopf's lemma and the strong 
		minimum principle for the fractional Laplacian. 
		In \cite[Proposition 2.7]{MR2677613} the authors 
		show the strong minimum principle and a generalized Hopf lemma
		for fractional harmonic functions. 		
		Whereas, in \cite{RS},  under the assumption 
		$\Omega$ is a smooth bounded domain,
		it is proven a Hopf lemma for weak solutions of 
		a  Dirichlet problem. See also \cite{MR3395749,MR3447732}. 
		For a Hopf lemma with mixed boundary condition, 
		see \cite{2016arXiv160701505B}. 
		
		\medskip
		
		Finally, Theorems 
		\ref{teo:mprincipleVis} and \ref{teo:Hopf} are 
		known for the fractional Laplacian 
		only for pointwise solutions, see \cite{Greco2014HopfsLA}. See
		also \cite{Iannizzotto2015} for $p=2$ and \cite{MR3474405}  for $p\neq 2.$
		Thus, our results generalize the results
		of \cite{Greco2014HopfsLA} in two way: for nonlinear 
		operators and weak solutions.

		\medskip

		To complete the introduction, we want to make 
		a little remark related to our result and the optimal 
		regularity of the Dirichlet problem. 
		Given $f\in L^\infty(\Omega),$ if 
		$\Omega\subset \R^N$ is a bounded smooth domain  and $u$ is a 
		weak solution of 
		\[
			(-\Delta_p)^s u = f(x) \quad \text{in } \Omega,
			\quad u=0 \quad\text{in }\Omega^c,
		\]
		then, by \cite[Theorem 1.1]{IMS}, 
		there is $\alpha=\alpha(N,s,p)\in(0,1)$ such that
		$u\in C^{\alpha}(\overline{\Omega}).$ 
		In fact that, we cannot expect more than
		$s-$H\"older continuity, see \cite[Section 3]{IMS}.
		
		Also, by Theorem \ref{teo:Hopf}, 
		we can deduce that $\alpha\le s.$ 
		Suppose that there exists a function 
		$c\in C(\overline{\Omega})$ such that
		$c\le0$ in $\Omega$ and $c(x)|u(x)|^{p-2}u(x)\le f(x),$ 
		(for instance, if $f\ge0$ we can take $c\equiv0$).
		Then $u$ is a weak super-solution of
		\[
			(-\Delta_p)^s u = c(x)|u(x)|^{p-2}u(x) \quad \text{in } \Omega.
		\]
		Thus, by Theorem \ref{teo:Hopf},  $\alpha\le s.$

\section{Preliminaries}\label{Preli}
	Let's start by introducing the notations and definitions 
	that we will use in this work. 
	We also gather some preliminaries properties which 
	will be useful in the forthcoming 
	sections.
	
	\medskip
	
	 If $t\in\R$ and $q>0,$ we will denote 
	$|t|^{q-1}t$ by $t^q$. For all function 
	$u\colon \Omega\to\R$ we define 
	\[
		u_+(x)\coloneqq\max\{u(x),0\} \quad\mbox{ and } \quad
		u_-(x)\coloneqq\max\{-u(x),0\},
	\]
	\[
		\Omega_+\coloneqq\{x\in\Omega\colon u(x)>0\}
		\quad\mbox{ and }\quad
		\Omega_-\coloneqq\{x\in\Omega\colon u(x)<0\}.
	\]

	\medskip
					
		
		Our next remark shows that $u_+$ and $u_-$
		belong to the same space as $u.$
		
		\begin{re}
			If $\mathcal{X}=\wsp$ or $\wwsp$ or $\mwsp,$ 
			and $u\in\mathcal{X}$  then $u_{+},u_{-}\in\mathcal{X}$
			owing to 
			\[
				|u_{-}(x)-u_{-}(y)|\le |u(x)-u(y)|
				\quad\mbox{ and }\quad
				|u_{+}(x)-u_{+}(y)|\le |u(x)-u(y)|
			\]
			for all $x,y\in\Omega.$
		\end{re}

		\medskip
	
		The proof of the following results can be found in \cite{IMS}.
	
		\begin{lem}{\cite[Lemma 2.8]{IMS}}\label{lema:nonlocalbehavior}
			Suppose $f\in L^1_{loc}(\Omega)$ and 
			$u\in\mwsp$ is a weak solution of $(-\Delta_p)^su=f$
			in $\Omega.$ Let $v\in L^1_{loc}(\mathbb{R}^N)$ be such that
			\[
				\dist(\supp(v),\Omega)>0 \mbox{ and } 
				\int_{\Omega^c}\dfrac{|v(x)|^{p-1}}{(1+|x|)^{N+sp}}
				dx<\infty,
			\]
			and define for a.e Lebesgue point $x\in\Omega$ of $u$
			\[
		 		h(x)=2\int_{\supp{v}}
		 		\dfrac{(u(x)-u(y)-v(y))^{p-1}-
		 		(u(x)-u(y))^{p-1}}{|x-y|^{N+sp}}dy.
			\]
			Then $u+v\in\mwsp$ and is a weak solution of 
			$(-\Delta_p)^s(u+v)=f+h$ in $\Omega.$
		\end{lem}
	
		\begin{teo}{\cite[Theorem 3.6]{IMS}}\label{teo:dist} 
			Let $\Omega$ be a bounded domain such that 
			$\partial \Omega$ is $C^{1,1},$ and 
			$\delta_{\Omega}(x)=\dist(x, \Omega^c).$ There
			exists $\rho=\rho(N,p,s,\Omega)$ such that 
			$\delta_{\Omega}^s$ is a weak solution of 
			$(-\Delta_p)^s\delta_{\Omega}^s=f$ in 
			$\Omega_{\rho}=\{x\in\Omega\colon 
			\delta_{\Omega}(x)<\rho\}$ for some 
			$f\in L^{\infty}(\Omega_{\rho}).$
		\end{teo}

		\begin{pro}{\cite[Proposition 2.10]{IMS}}\label{pro:cp}
			Let $\Omega$ be bounded, $u,v\in\mwsp$ satisfy 
			$u\ge v$ in $\Omega^c$ and for all $\varphi\in\wwsp,$
			$\varphi\ge0$
						\begin{align*}
				\int_{\R^{2N}}&
				\dfrac{(u(x)-u(y))^{p-1}(\varphi(x)-\varphi(y))}
				{|x-y|^{N+sp}}dxdy\ge\\
				&\qquad \int_{\R^{2N}}
				\dfrac{(v(x)-
				v(y))^{p-1}(\varphi(x)-\varphi(y))}
				{|x-y|^{N+sp}}dxdy.
			\end{align*}
			Then $u\ge v$ a.e. in $\Omega.$
		\end{pro}
	
		We also have a comparison principle for sub-solutions and 
		super-solutions of \eqref{eq:ecC}.
		\begin{pro}\label{pro:cp2}
			Let $\Omega$ be bounded, 
			$u,v\in \mwsp$ be nonnegative  super-solution and 
			sub-solution of \eqref{eq:ecC} in $\Omega$ respectively.  
			If $c(x)\le 0$ in $\Omega$ 
			and $u\ge v$ a.e. 
			in $\Omega^c$ then $u\ge v$ a.e. in $\Omega.$
		\end{pro}
	\begin{proof}
		We first observe that since $u,v\in\mwsp$ we have that 
		$(v-u)_+\in\wwsp.$  Then, using that $u,v$ are 
		super-solution and 
		sub-solution of \eqref{eq:ecC} in $\Omega$ respectively,
		we get
		\begin{align*}
			&\int_{\R^{2N}}\!\!\!\!\!\!\dfrac{(v(x)-v(y))^{p-1}
			-(u(x)-u(y))^{p-1}}{|x-y|^{N+sp}}
			((v(x)-u(x))_+-(v(y)-u(y))_+) dx dy\\
			&\le\int_{\Omega}c(x)(v(x)^{p-1}-u(x)^{p-1})
			(v(x)-u(x))_+ dx\le0.
		\end{align*}
		The proof follows by the argument of \cite[Lemma 9]{LL}.
	\end{proof} 
	
	Our next result is referred to the regularity of the weak solutions.
	
	\begin{lem}\label{lem:regularida}
		Let $\Omega\subset\R^N$ be smooth bounded domain and $c\in L^\infty(\Omega).$
		If $u\in\wwsp$ is a weak solution of \eqref{eq:ecC} then
		there is $\alpha\in(0,1)$ such that 
		$u\in C^{\alpha}(\overline{\Omega}).$  
	\end{lem}
	
	\begin{proof}
		By \cite[Lemma 2.3]{MR3530213} 
		and bootstrap argument, we have that
		$u\in L^\infty(\Omega).$ Therefore, by \cite[Theorem 1.1]{IMS},
		$u\in C^{\alpha}(\overline{\Omega})$ for some 
		$\alpha\in(0,1).$
	\end{proof}
	
	\subsection{Viscosity solution}
	    In he rest of this section, $\Omega$ is bounded 
	    open set with smooth boundary 
	    and $c\in C(\overline{\Omega}).$ 
	    
		Following \cite{KKL}, we define our notion of
		viscosity super-solution of \eqref{eq:ecC}. 
		We start to introduce some notation 
		\[	
			L_{s,p}(\R^N)=\left\{u\in L^{p-1}_{loc}(\R^N)\colon
			[u]_{s,p}<\infty\right\}.
		\]
		The set of critical points
		of a differential function $u$ and the distance  from the
		critical points are denoted by
		\[
			N_u\coloneqq\{x\in\Omega\colon\nabla u(x)=0\},
			\quad d_u(x)\coloneqq\mathrm{dist}(x,N_u),
		\]
		respectively. Let $D\subset\Omega$ be an open set. We denote
		the class of $C^2-$functions whose gradient and Hessian are
		controlled by $d_u$ as
		\[
			C^2_{\beta}(D)\coloneqq\left\{
			u\in C^2(\Omega)\colon \sup_{x\in D}
			\left(\dfrac{\min\{d_u(x),1\}^{\beta-1}}
			{|\nabla u(x)|}+\dfrac{|D^2u(x)|}{d_u(x)^{\beta-2}}
			\right)<\infty\right\}.
		\]
		
		\medskip
		
		We are now in condition to introduce our definition. 
		We say that a function $u\colon\mathbb{R}^N\to[-\infty,\infty]$
		is a viscosity super-solution of \eqref{eq:ecC} if it satisfies
		the following four assumptions:
		\begin{enumerate}
			\item[(VS1)] $u<\infty$ a.e. in $\mathbb{R}^N$ and
				$u>-\infty$ everywhere in $\Omega;$
			\item[(VS2)] $u$ is lower semicontinuous in $\Omega;$
			\item[(VS3)] If $\phi\in C^2(B_r(x_0))$ for some 
			$B_r(x_0)\subset\Omega$ such that $\phi(x_0)=u(x_0)$ and 
			$\phi\le u$ in $B_r(x_0),$ and one of the following holds
			\begin{enumerate}
				\item[(a)]$p>\nicefrac2{(2-s)}$ or 
					$\nabla \phi(x_0)\neq0;$
				\item[(b)]$1<p\le\nicefrac2{(2-s)};$
					$\nabla \phi(x_0)=0$ such that $x_0$ is an isolate
					critical point of $\phi,$ and 
					$\phi\in C^2_{\beta}(B_r(x_0))$ for some
					$\beta>\nicefrac{sp}{p-1};$ 
			\end{enumerate}
			then $(-\Delta_p)^s\phi_r(x_0)\ge c(x_0)u(x_0)^{p-1},$ where
			\begin{equation}\label{ec:fir}
				\phi_r(x)=
				\begin{cases}
					\phi &\text{ if } x\in B_r(x_0),\\
					u(x) &\text{otherwise};
				\end{cases}
			\end{equation}
			\item[(VS4)] $u_-\in L_{s,p}(\R^N).$   
		\end{enumerate}
		
		A function $u$ is a viscosity sub-solution of \eqref{eq:ecC}
		if $-u$ is a viscosity super-solution. Finally, $u$ is a 
		viscosity solution if it is both a viscosity
		super-solution and sub-solutions.
		
		\medskip
		
		To prove the following results, we borrow ideas and techniques
		of \cite[Proposition 11]{LL}. 
	
		\begin{lem}\label{lema:visco}
		 	Let $c\in C(\overline{\Omega}).$ 
		 	If $u\in \mwsp\cap C(\overline{\Omega})$
		 	is a weak super-solution of \eqref{eq:ecC} 
		 	such that $u\ge0$ in $\Omega^c$
		 	then $u$ is a viscosity super-solution of \eqref{eq:ecC}.
		\end{lem}
		\begin{proof}
			Let's observe that, by our assumptions, $u$ satisfies 
		 	(VS1),(VS2) and (VS4). Thus, we only need verify 
		 	property (VS3). We prove it by contradiction. 
		 	Suppose the conclusion in the lemma is
			false. Then there exist $x_0\in\Omega,$ and
			$\phi\in C^2(B_r(x_0))$ such that 
			\begin{itemize}
				\item $\phi(x_0)=u(x_0),$ and
					$\phi\le u \text{ in } B_r(x_0)\subset\Omega;$
				\item Either (a) or (b) in (VS3) holds;	
				\item $(-\Delta_p)^s\phi_r(x_0)< c(x_0)
					|\phi_r(x_0)|^{p-2}\phi_r(x_0)
					=c(x_0)|u(x_0)|^{p-2}u(x_0).$		
			\end{itemize}
			Then, by continuity (see \cite[Lemma 3.8]{KKL}), 
			there is $\delta\in(0,r)$ such that
			\[
					(-\Delta_p)^s\phi_r(x)< c(x)u(x)^{p-1}
			\]
			for all $x\in B_{\delta}(x_0).$
			
			By \cite[Lemma 3.9]{KKL}, there exist  $\theta>0,$ 
			$\rho\in(0,\nicefrac{\delta}2)$ and $\mu\in 
			C_0^2(B_{\nicefrac\rho2}(x_0))$ with $0\le\mu\le1$
			and $\mu(x_0)=1$ such that 
			$v=\phi_r+\theta\mu$ satisfies 
			\[
				\sup_{B_\rho(x_0)}|(-\Delta_p)^s\phi_r(x)-
				(-\Delta_p)^s v(x)|<
				\inf_{B_{\nicefrac{\delta}2}(x_0)}
				c(x)u(x)^{p-1}- (-\Delta_p)^s\phi_r(x).
			\]
			Then
			\[
				(-\Delta_p)^s v(x)< c(x)u(x)^{p-1}
			\]
			for all $x\in {B_\rho(x_0)}.$ 
			Then, $v=\phi_r 
			\le u$ in $B_{\rho}(x_0)^c$
			and by \cite[Lemma 10]{LL}, 
			for any $\varphi\in\widetilde{W}^{s,p}(B_\rho(x_0)),$ 
			$\varphi\ge0$
			\begin{align*}
				\int_{\R^{2N}}&
				\dfrac{(u(x)-u(y))^{p-1}(\varphi(x)-\varphi(y))}
				{|x-y|^{N+sp}}dxdy\ge\\
				&\qquad \int_{\R^{2N}}
				\dfrac{(v(x)-
				v(y))^{p-1}(\varphi(x)-\varphi(y))}
				{|x-y|^{N+sp}}dxdy.
			\end{align*}
				
		Therefore, by Proposition \ref{pro:cp}, $u\ge v$ in $B_\rho.$ Thus
		$u(x_0)=\phi_r(x_0)>\phi_r(x_0)+\theta =v(x_0)$ 
		which is a contradiction.
	\end{proof}

\section{Strong minimum principle}\label{SMP}

	Let us now  prove a strong minimum principle for 
	weak super-solutions of \eqref{eq:ecC}. To this end, we follow the 
	ideas in \cite{BF} and prove first the next logarithmic lemma 
	(see \cite[Lemma 1.3]{MR3542614}).
	
	\begin{lem}\label{lema:DKP}
		Let $c\in L^{1}_{loc}(\Omega),$ 
		and $u\in\mwsp$ be a weak super-solution of  
		\eqref{eq:ecC}. If   $u\geq 0$ a.e. in 
		$B_R(x_0)\subset\subset\Omega$ 
		then for any $B_r=B_r(x_0)\subset B_{\nicefrac{R}2}(x_0)$ 
		and $0<h<1$ we have that
		\begin{align*}
  			\int_{B_r^2}& \dfrac{1}{|x-y|^{N+sp}} 
  			\left| \log \left(\dfrac{u(x)+h}{u(y)+h}
  			\right)\right|^p \, d x d y 
  			\le \\
  			& Cr^{N-sp} \left\{h^{1-p}r^{sp}
  			\int_{B_{2r}^c} 
  			\dfrac{u_{-}(y)^{p-1}}{|y-x_0|^{N+sp}}\, d y + 1 
  			\right\} + \|c\|_{L^{1}(B_{2r})},
		\end{align*}
		where $C$ depends only on $N,s,$ and $p$.
	\end{lem}
	
	\begin{proof}
		Let $0<r<\nicefrac{R}{2},$ $0<h<1$ and 
		$\phi\in C_0^\infty( B_{\nicefrac{3r}2})$ be such that 	
		\[
		  0\le \phi \le 1, \quad \phi\equiv1 \text{ in } B_r 
		  \quad \text{ and } 
		  \quad|D\phi|<Cr^{-1} 
		  \text{ in } B_{\nicefrac{3r}2}\subset B_{R}. 
		\]
		Since $v=(u+h)^{1-p}\phi^p\in\wwsp$ 
		and $u$ is a super-solution of \eqref{eq:ecC},
		we have that
		\begin{equation}\label{eq:DKP}
  			\begin{aligned}
    				&\int_{B_{\nicefrac{3r}2}}c(x)
    				\frac{u^{p-1}(x)\phi^p(x)}{(u(x)+h)^{p-1}}d x\\ 
    				&\le\int_{\R^{2N}}
    				\dfrac{(u(x)-u(y))^{p-1}}{|x-y|^{N+ps}}\left(
   					\frac{\phi^p(x)}{(u(x)+h)^{p-1}}
   					-\frac{\phi^p(y)}{(u(y)+h)^{p-1}}
   				\right)dxdy. 
			\end{aligned}
		\end{equation}
		In the proof of Lemma 1.3 in \cite{MR3542614}, 
		it is showed that right  side of the above inequality is bounded 
		by
		\begin{align*}
  			\ Cr^{N-sp}& \left\{h^{1-p}r^{sp}\int_{ (B_{2r})^c} 
  			\dfrac{(u_{-}(y))^{p-1}}{|y-x_0|^{N+sp}}\, d y+1 \right\}\\
  			&\qquad\qquad -\int_{(B_r)^2} \dfrac{1}{|x-y|^{N+sp}} \left| 
  			\log\left(\dfrac{u(x)+h}{u(y)+h}\right)\right|^p \, dxdy,
		\end{align*}
		where $C$ depends only on $N,s,$ and $p.$ 
		Then, by \eqref{eq:DKP} and using that 
		$0\le u^{p-1}(u+h)^{1-p}\phi^p\le1$ in 
		$B_{\nicefrac{3r}{2}},$ the lemma holds.
	\end{proof}
	
	\begin{lem}\label{lema:mprinc}
		Let {$c \in L^{1}_{loc}(\Omega)$} be a non-positive function 
		and  $u$ be a weak super-solution of  
		\eqref{eq:ecC}. Then,
		\begin{enumerate}
			\item[(a)] If $u\ge0$ a.e. in $\Omega^c$ and 
				$u=0$ a.e. in $\Omega$ then $u=0$ a.e. 
				in $\R^N.$ 
			\item[(b)] If $\Omega$ is bounded,  
				and $u\ge0$ a.e. in $\Omega^c$ then 
				$u\ge0$ a.e. in $\Omega.$ 
		\end{enumerate} 
		
	\end{lem}
	\begin{proof}
		First, we prove (a). Let $\varphi\in C^{\infty}_0(\Omega)$ be 
		non-negative function, then
		\begin{align*}
			0&\le\int_{\R^{2N}}\dfrac{(u(x)-u(y))^{p-1}
			(\varphi(x)-\varphi(y))}{|x-y|^{N+sp}}dxdy\\
			&=-2\int_{\Omega^c\times\Omega}
			\dfrac{u(y)^{p-1}
			\varphi(x)}{|x-y|^{N+sp}}dxdy
		\end{align*}
		due to $u=0$ a.e. in $\Omega.$
		Thus, since $u\ge0$ a.e. in $\Omega^c$ then
		$u=0$ a.e. in $\Omega^c.$ Hence $u=0$ a.e. in $\R^N.$
		
		Now we prove (b). 
		Since $u\in\mwsp$ and $u\ge 0$ in $\Omega^c$  
		we have that $u_{-}\in\wwsp.$ Then 
		\begin{align*}
			0&\le -\int_{\Omega}c(x)(u_{-}(x))^p dx=
			\int_{\Omega}c(x)(u(x))^{p-1}u_{-}(x) dx\\
			&\le \int_{\R^{2N}}
			\dfrac{(u(x)-u(y))^{p-1}(u_{-}(x)-u_{-}(y))}{|x-y|^{N+sp}}  dx
		\end{align*}
		owing to $c(x)\le0$ in $\Omega$ and 
		$u$ is a weak super-solution of  \eqref{eq:ecC}.
		Observe that
		\begin{align*}
			(u(x)-u(y))^{p-1}&(u_{-}(x)-u_{-}(y))\le\\
			&\le
			\begin{cases}
				-(u_{-}(x)-u_{-}(y))^p & \mbox{if } x,y\in \Omega_{-},\\
				-(u_{-}(x)+u_{+}(y))^{p-1}u_{-}(x) & \mbox{if } x\in 
				\Omega_{-}, 
				y\in \Omega_{-}^c 
			\end{cases}
		\end{align*}
		consequently
		\begin{align*}
			0&\le \int_{\R^{2N}}
			\dfrac{(u(x)-u(y))^{p-1}(u_{-}(x)-
			u_{-}(y))}{|x-y|^{N+sp}}  dx\\
				&\le-\int_{\Omega_{-}^2}
				\dfrac{|u_{-}(x)-u_{-}(y)|^p}{|x-y|^{N+sp}}
				\, dxdy
					-2\int_{\Omega_{-}\times \Omega_{-}^c}
					\dfrac{(u_{-}(x)+
					u_{+}(y))^{p-1}u_{-}(x)}{|x-y|^{N+sp}}\, 
					dxdy\\
				&\le0.
		\end{align*}
		Therefore $u_{-}\equiv0$ a.e. in $\R^N.$ Then in both 
		cases we have that $u\ge0$ a.e. in $\R^N.$ 
	\end{proof}
	
	Now, we prove our strong minimum principle under the 
	assumption that $\Omega$ is connected.	
	
	\begin{lem}\label{lema:mprinciple} 
		Let {$c \in L^{1}_{loc}(\Omega)$} be a non-positive function 
		and  $u\in\mwsp$ be a weak super-solution of  
		\eqref{eq:ecC}. 
		\begin{enumerate}
			\item If $\Omega$ is bounded and connected,  
				and $u\ge0$ a.e in $\Omega^c$ then either 
				$u>0$ { a.e.} in $\Omega$ or $u=0$ a.e. 
				in $\Omega.$ 
			\item  If $\Omega$ is connected, 
				$u\ge0$ a.e in $\mathbb{R}^N$ then either 
				$u>0$ {a.e.} in $\Omega$ or $u=0$ a.e. in $\Omega.$ 
		\end{enumerate}  		
	\end{lem}
	\begin{proof}
		By Lemma \ref{lema:mprinc}, $u\ge0$ a.e. in $\mathbb{R}^N$.
		
		Proceeding as in the proof of Theorem A.1 
		in \cite{BF} and using Lemma \ref{lema:DKP}, we have that
		If $\Omega$ is bounded and connected $u\not=0$ a.e. in $\Omega$, 
		then $u>0$ a.e. in $\Omega.$
		
		\medskip
		
		If $\Omega$ is unbounded and connected, then there is a 
		sequence of bounded connected open sets 
		$\{\Omega_n\}_{n\in\mathbb{N}}$ such that 
		$\Omega_n\subset\Omega_{n+1}\subset\Omega$  
		for all $n\in\mathbb{N}$ and 
		$\Omega=\cup_{n\in\mathbb{N}}\Omega_{n}.$ If $u\not = 0$
		a.e. in $\Omega$ then there is 
		$n_{0}\in\mathbb{N}$ such that $u\not= 0$
		a.e. in $\Omega_n$ for all $n\ge n_0.$ Thus $u>0$ a.e. 
		in $\Omega_n$ for all 
		$n\ge n_0,$ since for all $n\ge n_0$ we have that $\Omega_n$ 
		is a bounded conected open set, $u$ is be a nonnegative 
		weak super-solution of $(-\Delta)_{p}^s u = c(x)u^{p-1}$ 
		in $\Omega_n$ and $u\not=0$ a.e. in $\Omega_n.$ Therefore $u>0$
		a.e in $\Omega.$

	\end{proof}
	
		In fact, as our operator is non-local,  
		we do not need to assume that the domain is connected.
	
	\begin{lem}\label{lema:mprinciple2} 
		Let {$c \in L^{1}_{loc}(\Omega)$} be a non-positive function 
		and  $u\in\mwsp$ be a weak super-solution of  
		\eqref{eq:ecC}. 		
		\begin{enumerate}
			\item If $\Omega$ is bounded,  
				and $u\ge0$ a.e. in $\Omega^c$ then either 
				$u>0$ {a.e.} in $\Omega$ or $u=0$ a.e. in $\Omega.$
			\item  If  $u\ge0$ a.e. in $\mathbb{R}^N$ then either 
				$u>0$ {a.e.} in $\Omega$ or $u=0$ a.e. in $\Omega.$
		\end{enumerate}  				
	\end{lem}
	
	\begin{proof}
		By Lemma \ref{lema:mprinciple}, we only need to 
		show that $u\neq0$ a.e in $\Omega$ if only if 
		$u\not\neq 0$  a.e. in all connected components of $\Omega.$
		That is, we only need  to show that if
		$u\not\equiv0$ in $\Omega$ then $u\not\equiv0$ 
		in all connected components of $\Omega.$
		
		Suppose, on the contrary, that there is a connected component 
		$U$ of $\Omega$ such that $u=0$ a.e. in $U.$ 
		Since $u$ is a weak super-solution of \eqref{eq:ecC},
		it follows from Lemma \ref{lema:mprinc} that $u\ge0$ in $\R^N.$
		Moreover, for any nonnegative function 
		$\varphi \in \widetilde{W}^{s,p}(\Omega)$ we get
		\begin{align*}
			0&\le\int_{\R^{2N}}\dfrac{(u(x)-u(y))^{p-1}
			(\varphi(x)-\varphi(y))}
			{|x-y|^{N+sp}}
			dxdy\\
			&=-2\int_{U}\int_{U^c}
			\dfrac{(u(x))^{p-1}
			\varphi(y)}
			{|x-y|^{N+sp}}
			dxdy
		\end{align*}
		due to $u=0$ a.e. in $U.$ Then $u=0$ a.e. in $U^c,$
		that is $u=0$ a.e. in $\R^N,$ 
		which is a contradiction to the fact that
		$u\not=0$ a.e. in $\Omega.$
	\end{proof}

	Then, by Lemmas \ref{lema:mprinc} and \ref{lema:mprinciple2}, we get
	Theorem \ref{teo:mprinciple}. 
	
	\medskip

	To conclude this section, we prove Theorem 
	\ref{teo:mprincipleVis}. The key of the proof is 
	Theorem \ref{teo:mprinciple} and the next result.
		
	\begin{lem}\label{lema:positivity}
		Let $c\in C(\overline{B_R(x_0)})$ 
		and $u\in\mathcal{W}^{s,p}(B_R(x_0))\cap C(\overline{B_R(x_0)})$ 
		be a weak super-solution of 
		\begin{equation}\label{eq:ecCB}
			(-\Delta_p)^su=c(x)u^{p-1} 
			\quad \text{ in } B_R(x_0).
		\end{equation} 
		If $u\ge 0$ in $B_R(x_0)^c$ then either $u>0$ 
		in $B_R(x_0)$ or $u=0$ a.e. in $\R^N.$ 
	\end{lem}
	\begin{proof}
		We will show that if there is $x_\star\in B_R=B_R(x_0)$
		such that $u(x_\star)=0$ then $u=0$ a.e. in $\R^N.$
		
		We start observing that, by Lemma \ref{lema:mprinc}, 
		$u\ge 0$ a.e. in $B_R.$ Moreover, by 
		Theorem \ref{teo:mprinciple},
		either $u>0$ a.e. in $B_R$ or $u=0$ a.e in $\R^N.$ 
	
		On the other hand, by Lemma \ref{lema:visco}, $u$ is a
		viscosity super-solution of \eqref{eq:ecCB}. Then, since
		$u\ge0$ in $B_R,$ for any $\varepsilon>0$ and 
		$\beta>\max\{2,\nicefrac2{2-s}\}$ the function
		\[
			\phi^\varepsilon=-\varepsilon|x-x_\star|^\beta
		\]
		is an admissible test function. Therefore
		\[
			(-\Delta_p)^s\phi^\varepsilon_r(x_{\star})\ge 
			c(x_{\star})\phi^\varepsilon_r(x_{\star})^{p-1}=0		
		\]
		for some $r\in (0,R-|x_{\star}-x_0|).$ See \eqref{ec:fir}
		for the definition of $\phi^\varepsilon_r.$   
		
		Then
		\begin{equation}\label{eq:auxposvisco}
			0\le
			\varepsilon^{p-1}\int_{B_r(x_\star)}
			|y-x_{\star}|^{\beta(p-1)-N-ps} dy-
			\int_{B_r(x_0)^c}	
			\dfrac{u(y)^{p-1}}{|y-x_{\star}|^{N+ps}}
			dy.
		\end{equation}
		Since $\beta>\max\{2,\nicefrac2{2-s}\},$ we get
		\[
			\varepsilon^{p-1}\int_{B_r(x_\star)}
			|y-x_{\star}|^{\beta(p-1)-N-ps} dy=
			\dfrac{\varepsilon^{p-1}\omega_N}{\beta(p-1)-ps}
			r^{\beta(p-1)-ps}\to0\quad\text{as }\varepsilon\to0,
		\]
		where $\omega_N$ denotes the volume of the unit ball in 
		$\R^N.$ Then, by \eqref{eq:auxposvisco}, we get
		$u=0$ a.e. in $B_r(x_\star)^c.$ 
		Therefore $u=0$ a.e. $\R^N.$
	\end{proof}
	
	Now we can prove Theorem \ref{teo:mprincipleVis}.
	
	\begin{proof}[Proof of Theorem \ref{teo:mprincipleVis}]
		By Theorem \ref{teo:mprinciple}, we have that 
		$u>0$ a.e. in $\Omega$ or $u=0$ $a.e$ in $\R^N.$ Suppose that
		there is $x_0\in\Omega$ such that $u(x_0)=0.$ Since 
		$\Omega$ is open, $c,u\in C(\Omega),$ there is $R>0$ such
		that $B_R(x_0)\subset\Omega$ and 
		$c,u\in C(\overline{B_R(x_0)}).$ Moreover, since $u$ is a weak
		super-solution of \eqref{eq:ecC}, we have that
		$u$ is a weak super-solution of 
		\[
			(-\Delta_p)^su=c(x)u^{p-1} 
			\quad \text{ in } B_R(x_0).
		\]
		Then, by Lemma \ref{lema:positivity}, $u=0$ a.e. in $\R^N$
		since $u(x_0)=0.$ Therefore $u>0$ in $\Omega$ or $u=0$ in 
		$\Omega.$
	\end{proof}	
	\begin{re}\label{re:notnecassump2}
		In the case $u\ge0$ a.e. in $\R^n,$ 
		the non-positivity assumption over 
		the function $c$ is not necessary, see
		Remark \ref{re:notnecassump}.
	\end{re}
\section{A Hopf lemma}\label{Hopf}
	First, we can show the Hopf's lemma in a ball.

	\begin{lem}\label{lema:HopfBola}
	 	Let $B$ be a ball in $\R^N$ of radius $R>0$, 
	 	$c,u\in C(\overline{B}),$ u 
		be a  weak super-solution of 
		\[
			(-\Delta_p)^su=c(x)u^{p-1} \quad \text{ in } B
		\]
		 and 
		$\delta(x)=\dist(x,B^c)$. 
		If $u\ge0$ a.e. in $\R^N$ 
		or $c(x)\le 0$ in $B$ and $u\ge0$ a.e. in $B^c$  
		then either $u\equiv0$ a.e. in $\R^N$ or 
		\begin{equation}\label{eq:Hopfbola}
				\liminf_{B\ni x\to x_0}\dfrac{u(x)}{\delta^s(x)}>0
		\end{equation}
		for all $x_0\in\partial B.$ 
	\end{lem}
		 
	\begin{proof}
		By Theorem \ref{teo:mprincipleVis} and Remark
		\ref{re:notnecassump2}, we have 
		that either $u=0$ a.e in $\R^N$ or $u>0$ in $B.$ Suppose
		$u\not\equiv 0$ in $B,$ then 
		we want to show that \eqref{eq:Hopfbola} 
		holds true for all $x_0\in\partial B.$
		
		\medskip
		
		By Theorem \ref{teo:dist}, there exists 
		$\rho=\rho(N,p,s,B)>0$ such that $\delta^s$ 
		is a weak solution
		of $(-\Delta_p)^s\delta^s=f$ in 
		$B_\rho=\{x\in B\colon \delta(x)<\rho\}$ 
		for some $f\in L^{\infty}(B_\rho).$
		Let $K\subset\subset B_\rho^c\cap B$ be 
		a closed ball and $\alpha>0$ 
		be a constant (to be determined later).
		Owing to Lemma \ref{lema:nonlocalbehavior},  
		$w=\delta^s+\alpha \chi_K$ is a weak solution of 
		$(-\Delta_p)^sw=f+h_{\alpha}$ in $B_\rho$ where
		\[
			h_{\alpha}(x)=2\int_{K} 
			\dfrac{(\delta^s(x)-\delta^s(y)-\alpha)^{p-1}
			-(\delta^s(x)-\delta^s(y))^{p-1}}
			{|x-y|^{N+sp}}\, dxdy
		\]
		for a.e. $x\in B_\rho.$ Since 
		$u\in L^{\infty}(\overline{B})$ and 
		$\dist(K,B_{\rho})>0,$ it is clear that
		\[
			h_{\alpha}(x)\to-\infty \quad 
			\mbox{uniformly in } \overline{B_{\rho}} 
			\quad \mbox{ as }  
			\alpha \to \infty.
		\] 
		Then, we choose $\alpha$ large enough such that 
		\begin{equation}\label{eq:desaux1}
			\sup \{f(x)+ h_{\alpha}(x)\colon x\in B_{\rho} \}
			\le \inf\{c(x)(u(x))^{p-1}\colon x\in B_{\rho}\}.
		\end{equation}

		Let $\varepsilon\in(0,1)$ be such that 
		\[
			\varepsilon(R^s+\alpha)<\inf\{u(x)\colon x\in B, 
				\delta(x)\ge\rho\}.
		\] 
		Thus $v=\varepsilon w\le u$ in
		$B_{\rho}^c$ and using \eqref{eq:desaux1} we have that
		for all $\varphi\in\wwsp,$
		$\varphi\ge0$ 
		\begin{align*}
			\int_{\R^{2N}}
			&\dfrac{(v(x)-v(y))^{p-1}(\varphi(x)-\varphi(y))}
			{|x-y|^{N+sp}}dxdy\\&\le
			\int_{\R^{2N}}
			\dfrac{(u(x)-u(y))^{p-1}(\varphi(x)-\varphi(y))}
			{|x-y|^{N+sp}}dxdy.
		\end{align*}
		 Then, by Proposition \ref{pro:cp}, $v\le u$ in
		$B_\rho.$ Therefore
		\[
			\varepsilon\le\dfrac{u(x)}{\delta^s(x)}
			\quad \forall x\in B_\rho.
		\]
		Thus \eqref{eq:Hopfbola} holds true for all 
		$x_0\in\partial B.$
	\end{proof}	
	
	To conclude this section, we show our Hopf Lemma.
	
	\begin{proof}[Proof of Theorem \ref{teo:Hopf}]
		By Theorem \ref{teo:mprincipleVis} and Remark
		\ref{re:notnecassump2}, we have 
		that either $u=0$ a.e. in $\R^N$ or $u>0$ in $\Omega.$ Suppose
		$u\not\equiv 0$ in $\Omega.$ Since $\Omega$ 
		satisfies the interior ball condition in 
		$x_0\in\partial\Omega,$ there is a ball $B\subset\Omega$ 
		such that $x_0\in\partial B.$ Then, 
		by Lemma \ref{lema:HopfBola}, \eqref{eq:Hopf}
		holds true.
	\end{proof}

\subsection*{Aknowledgements}
	L. M. Del Pezzo was partially supported by  
	CONICET PIP 5478/1438  (Argentina) 
	and A. Quaas was partially supported 
	by Fondecyt grant No. 1110210 and Basal CMM UChile. 
\bibliographystyle{siam}
\bibliography{Biblio}

\end{document}